\documentclass{jT}
\usepackage{amssymb,amsmath,latexsym,amsthm}
\numberwithin{equation}{section}
\usepackage{epic}

\newtheorem{Theorem}{Theorem}
\newtheorem{Proposition}[Theorem]{Proposition}
\newtheorem{Lemma}[Theorem]{Lemma}

\theoremstyle{definition}
\newtheorem*{Definition}{Definition}
\newtheorem*{Remark}{Remark}
\newtheorem*{Example}{Example}
\newtheorem*{acknowledgements}{Acknowledgements}

\newcommand{\R}{\mathbb R}
\newcommand{\Z}{\mathbb Z}

\newcommand{\N}{\mathbb N}

\ifx\undefined\fiverm \font\fiverm=cmr5 \fi

\input prepictex
\input pictex
\input postpictex

\begin{document}

\def\refname{\centerline{Bibliography}}

\title[Tong's Spectrum for Rosen Continued Fractions]{Tong's Spectrum for Rosen Continued Fractions}

\author[Cornelis {\sc Kraaikamp}]{{\sc Cornelis} KRAAIKAMP}
\address{Cornelis {\sc Kraaikamp}\\
EWI, Delft University of Technology, \\
Mekelweg 4, 2628 CD Delft \\
the Netherlands }
\email{c.kraaikamp@ewi.tudelft.nl}

\author[Thomas A. {\sc Schmidt}]{{\sc Thomas A.}
SCHMIDT \footnote{The second author was supported by NWO
Bezoekersbeurs { B 61-505}.} \break}
\address{Thomas A. {\sc Schmidt}\\
Oregon State University\\
Corvallis, OR 97331\\
USA}
\email{toms@math.orst.edu}

\author[Ionica {\sc Smeets}]{{\sc Ionica} SMEETS}
\address{Ionica {\sc Smeets}\\
Mathematical Institute\\
Leiden University\\
Niels Bohrweg 1, 2333 CA Leiden\\
 the Netherlands}
\email{smeets@math.leidenuniv.nl}

\date{01/09/2006}
\maketitle

\begin{resume}
Dans les ann\'ees 90, J.C.~Tong a donn\'e une borne
sup\'e\-rieu\-re optimale pour le minimum de $k$ coefficients
d'approxi\-mation cons\'e\-cutifs  dans le cas des fractions continues \`a
l'entier le plus proche. Nous g\'en\'eralisons ce type de r\'esultat aux fractions conti\-nues de Rosen. Celles-ci
constituent une famille  infinie d'algorith\-mes de d\'eveloppement en
fractions continues, o\`u  les quotients partiels sont certains entiers alg\'ebriques r\'eels.
Pour chacun de ces  algo\-rithmes nous d\'eterminons la borne
sup\'erieure optimale de la valeur  mini\-male des coefficients d'approximation 
pris en nombres cons\'ecutifs appropri\'es. Nous
donnons aussi des r\'esultats m\'etri\-ques pour des  plages de ``mauvaises" approximations successives de grande lon\-gueur.
\end{resume}
\begin{abstr}
In the 1990s, J.C.~Tong gave a sharp upper bound on the minimum of
$k$ consecutive approximation constants for the nearest integer
continued fractions.  We generalize this to the case of
approximation by Rosen continued fraction expansions. The Rosen
fractions are an infinite set of continued fraction algorithms,
each giving expansions of real numbers in terms of certain
algebraic integers. For each, we give a best possible upper bound
for the minimum in appropriate consecutive blocks of approximation
coefficients. We also obtain metrical results for large blocks of
``bad'' approximations.
\end{abstr}
\section{Introduction}
It is well-known that every $x\in [0,1)\setminus \mathbb{Q}$ has a
unique (regular) continued fraction expansion of the form
\begin{equation}\label{def: RCF}
x=\frac{\displaystyle{1}}{\displaystyle{a_1}
+\frac{\displaystyle{1}}{\displaystyle{a_2}+\ldots
+\frac{\displaystyle{1}}{\displaystyle{a_n}+\ldots}}} = [\,
a_1,\, a_2,  \ldots ,\, a_n, \ldots].
\end{equation}\goodbreak 
Here the partial quotients $a_n$ are positive integers for $n\geq
1$. Finite truncation in~(\ref{def: RCF}) yields the convergents
$p_n/q_n$ of $x$, i.e., for $n\geq 1$
$$
\frac{p_n}{q_n}=\frac{\displaystyle{1}}{\displaystyle{a_1}
+\frac{\displaystyle{1}}{\displaystyle{a_2}+\ldots
+\frac{\displaystyle{1}}{\displaystyle{a_n}}}} = [\, a_1,\,
a_2,\ldots ,\, a_n],
$$
and throughout it is assumed that $p_n/q_n$ is in its lowest
terms. Note that (\ref{def: RCF}) is a shorthand for
$\lim_{n\to\infty} p_n/q_n=x$.

Underlying the regular continued fraction (RCF) expansion
(\ref{def: RCF}) is the map $T\, :\, [0,1)~\rightarrow~[0,1)$,
defined by
$$
T(x)\, = \frac{1}{x} \!\!\! \mod 1 \,= \frac{1}{x}-\left\lfloor
\frac{1}{x}\right\rfloor ,\, x\ne 0; \quad T(0)\, =\, 0.
$$
Here $\left\lfloor \frac{1}{x}\right\rfloor$ denotes the integer
part of $\frac{1}{x}$. The RCF-convergents of $x\in [0,1)\setminus
\mathbb{Q}$ have strong approximation properties. We mention here
that
$$
\left| x -\frac{p_n}{q_n}\right| < \frac{1}{q_n^2},\quad \text{for
$n\geq 0$,}
$$
which implies, together with the well-known recurrence relations
for the $p_n$ and $q_n$, that the rate of convergence of $p_n/q_n$ to $x$
is exponential (see e.g.\ \cite{[DK]}). One thus defines the
approximation coefficients $\theta_n(x)$ of $x$ by $\theta_n =
\theta_n(x)= q_n^2 \left| x - p_n/q_n \right|$, $n\geq 0$. We
usually suppress the dependence on $x$ in our notation.

For the RCF-expansion we have the following classical theorems by
Borel (1905) and Hurwitz (1891) about the quality of the
approximations.
\begin{Theorem}\emph{\textbf{(Borel)}} For every irrational
number $x$, and every $n\geq 1$
$$
\min\{\theta_{n-1},\theta_{n},\theta_{n+1}\} <
\displaystyle{\frac{1}{\sqrt{5}}}.
$$
The constant $1/\sqrt{5}$ is best possible.
\end{Theorem}
Borel's result, together with a yet earlier result by Legendre
\cite{[L]}, which states that if $p, q\in \Z$, $q>0$, and $\gcd
(p,q)=1$, then
$$
\left| x-\frac{p}{q}\right| < \frac{1}{2q^2}\quad
\textrm{ implies that}\quad \left(
\begin{array}{c}
p \\
q
\end{array}\right) = \left( \begin{array}{c}
p_n\\
q_n
\end{array}\right) ,\quad \text{for some $n\geq 0$},
$$
implies the following result by Hurwitz.
\begin{Theorem}\emph{\textbf{(Hurwitz)}}\label{Hurwitz} For every
irrational number $x$ there exist infinitely many pairs of
integers $p$ and $q$, such that
$$
\left| x -\frac{p}{q}\right| < \frac{1}{\sqrt{5}}\frac{1}{q^2}.
$$
The constant $1/\sqrt{5}$ is best possible.
\end{Theorem}
By removing all irrational numbers which are equivalent to the
`golden mean' $g=\frac{1}{2}(\sqrt{5}-1)$ (i.e., those irrationals
whose RCF-expansion consists of $1$s from some moment on), we have
that
$$
\left| x -\frac{p}{q}\right| < \frac{1}{\sqrt{8}}\frac{1}{q^2},
$$
for infinitely many pairs of integers $p$ and $q$. These constants
$1/\sqrt{5}$ and $1/\sqrt{8}$ are the first two points in the
so-called \emph{Markoff spectrum}; see \cite{[CF]} or \cite{[B]}
for further information on this spectrum, and the related Lagrange
spectrum.

Note that the theorem of Borel does not suffice to prove Hurwitz's
theorem; one needs Legendre's result to rule out the existence of
rationals $p/q$ which are not RCF-convergents, but still satisfy
$|x -p/q|< 1/(\sqrt{5}\, q^2)$.

In \cite{[T1],[T2]}, Tong generalized Borel's results to the
nearest integer continued fraction expansion (NICF). These are
continued fractions of the form
$$
x=\frac{\displaystyle{\varepsilon_1}}{\displaystyle{b_1}
+\frac{\displaystyle{\varepsilon_2}}{\displaystyle{b_2}+\ldots
+\frac{\displaystyle{\varepsilon_n}}{\displaystyle{b_n}+\ldots}}},
$$
generated by the operator $T_\frac12\, :\,
\left[-\frac12,\frac12\right)\rightarrow
\left[-\frac12,\frac12\right)$, defined by
\begin{equation}\label{nicf-map}
T_\frac12(x)\, = \frac{\varepsilon}{x} - \left\lfloor
\frac{\varepsilon }{x}  + \frac12
\right\rfloor ,\, x\ne 0; \quad T(0)\, =\, 0,
\end{equation}
where $\varepsilon$ denotes the sign of $x$. Since the
NICF-expansion of any number $x$ can be obtained from the
RCF-expansion via a process called singularization (see
\cite{[DK]} or \cite{[IK]} for details), the sequence of
NICF-convergents $(r_k/s_k)_{k\geq 0}$ is a subsequence
$(p_n/q_n)_{n\geq 0}$ of the sequence of RCF-conver\-gents  of
$x$. Due to this, the approximation by NICF-convergents is faster;
see e.g.\ \cite{[A]}, or \cite{[IK]}. In \cite{[BJW]} it was shown
that the approximation by NICF-convergents is also closer; for
almost all $x$ one has that
$$
\lim_{k\to \infty} \frac{1}{k}\sum_{i=0}^{k-1}\vartheta_k=
\frac{\sqrt{5}-2}{2 \log G}= 0.24528\dots
$$
 \text{ whereas}
 $$ \lim_{n\to
\infty} \frac{1}{n}\sum_{i=0}^{k-1}\theta_i= \frac{1}{4\log 2} =
0.36067\dots ,
$$

where $\vartheta_k=\vartheta_k(x)=s_k^2\left| x-r_k/s_k\right|$ is
the $k$th NICF-approximation coefficient of $x$, and $G=g+1$.\\
In contrast to this, it was shown in \cite{[JK]} that for almost
every $x$ there are infinitely many arbitrary large blocks of
NICF-approximation coefficients
$\vartheta_{n-1},\ldots,\vartheta_{n+k}$, which are all larger
than $1/\sqrt{5}$. In spite of this, it is also shown in
\cite{[JK]} that for \emph{all} irrational numbers $x$ there exist
infinitely many $k$ for which $\vartheta_k<1/\sqrt{5}$.

In~\cite{[T1]} and~\cite{[T2]}, Tong sharpened the results from
\cite{[JK]}, by showing that for the NICF there exists a
`pre-spectrum,' i.e., there exists a sequence of constants
$(c_k)_{k\geq 1}$, monotonically decreasing to $1/\sqrt{5}$, such
that for all irrational numbers $x$ the minimum of any block of
$k+2$ consecutive NICF-approximation coefficients is smaller than
$c_k$.
\begin{Theorem}\textbf{\emph{(Tong)}} For every irrational number $x$
and all positive integers $n$ and $k$ one has
$$
\min\{\vartheta_{n-1},\vartheta_{n},\ldots,\vartheta_{n+k}\} <
\frac{1}{\sqrt{5}} +
\frac{1}{\sqrt{5}}\left(\frac{3-\sqrt{5}}{2}\right)^{2k+3}.
$$
The constant $c_k=\frac{1}{\sqrt{5}} +
\frac{1}{\sqrt{5}}\left(\frac{3-\sqrt{5}}{2}\right)^{2k+3}$ is
best possible.
\end{Theorem}

In~\cite{[HK]}, Hartono and Kraaikamp showed how Tong's result
follows from a geometrical approach based on the natural extension
of the NICF. We further this approach to find Tong's spectrum for
an infinite family of continued fractions generalizing the NICF;
these \emph{Rosen fractions} are briefly described in the next
section.

Although the appropriate terms are only defined in the following
section, the reader may wish to compare Tong's Theorem with the
following, whose proof appears in Section ~\ref{sec:
even-indices}.   (The constants $\tau_{k}$ are given in the
statement of Theorem \ref{th: flushing}.)

\begin{Theorem}
\label{th: tong even} Fix an even $q = 2p$, with $p \ge 2$.  For
every $G_q$-irrational number $x$ and all positive $n$ and $k$,
one has
$$
\min \{ \Theta_{n-1},\Theta_n,\ldots,\Theta_{n+k(p-1)}\} <
\frac{-\tau_{k-1}}{1+(\lambda-1)\tau_{k-1}}.
$$
The constant $c_{k-1}=\frac{-\tau_{k-1}}{1+(\lambda-1)\tau_{k-1}}$
is best possible.
\end{Theorem}
We prove an analogous result for all odd indices of these $G_{q}$
in Section~\ref{sec: odd-indices}. In both cases, we prove a
Borel-type result.  Furthermore, our approach allows us to give
various metric results.

\section{Rosen continued fractions}\label{sec:Rosen fractions}
In 1954, David Rosen (see~\cite{[R]}) introduced a family of
continued fractions now bearing his name. The Rosen fractions form
an infinite family of continued fractions generalizing the NICF.
Although Rosen introduced his continued fractions to study certain
Fuchsian groups, we are only concerned with
their Diophantine approximation properties.

Define $\lambda_q = 2 \cos \frac{\pi}{q}$ for each $q \in
\{3,4,\dots\}$. For $q$ fixed,  to simplify notation we usually
write $\lambda$ for $\lambda_q$. For each $q$ the \emph{Rosen} or
\emph{$\lambda$-expansion} ($\lambda$CF) of $x$ is found by using
the map $f_q \, :\,
\left[-\frac{\lambda}{2},\frac{\lambda}{2}\right)\rightarrow
\left[-\frac{\lambda}{2},\frac{\lambda}{2}\right)$, defined by
\begin{equation}\label{eq: def LCF operator}
f_q(x)=\frac{\varepsilon}{x} - \lambda r(x),\, x\ne 0;
\quad f_q(0)\, =\, 0,
\end{equation}
where $\displaystyle{r(x)=\left\lfloor \frac{\varepsilon}{\lambda
x} + \frac12 \right\rfloor}$ and $\frac{\varepsilon}{x}=1/|x|$. We usually write $r$ instead of
$r(x)$. Since $\lambda_3=1$, we see that for $q=3$ the map $f_q$
is the NICF-operator $T_{\frac{1}{2}}$ from (\ref{nicf-map}). For
$x\in [-\lambda /2, \lambda /2)$, the map $f_q$ yields a continued
fraction of the form
$$
\begin{array}{lllll}
x= \frac{\displaystyle{\varepsilon_1}}{\displaystyle{r_1\lambda}+
\frac{\displaystyle{\varepsilon_2}}{\displaystyle{r_2\lambda}+\ldots
+\frac{\displaystyle{\varepsilon_n}}{\displaystyle{r_n\lambda}+\ldots}}}
=:  [\varepsilon_1:r_1,\varepsilon_2:r_2,\ldots,
\varepsilon_n:r_n,\ldots ],
\end{array}
$$
where $\varepsilon_i \in \{\pm 1\}$ and $r_i \in \N$. As usual,
finite truncations yield the convergents $R_n/S_n$, for $n\geq 0$,
i.e., $R_n/S_n=[\varepsilon_1:r_1,\varepsilon_2:r_2,\ldots,
\varepsilon_n:r_n]$. The (Rosen) approximation coefficients of $x$
are defined by
$$
\Theta_n=\Theta_n(x)=S_n^2\left| \, x-\frac{R_n}{S_n}\right|
,\quad \text{for $n\geq 0$}.
$$
For $x\in [-\lambda /2,\lambda /2)$, we define the \emph{future}
($t_n$) and the \emph{past} ($v_n$) of $x$ at time $n$ by
$$
t_n =
[\varepsilon_{n+1}:r_{n+1},\varepsilon_{n+2}:r_{n+2},\ldots],\quad
v_n = [1:r_n,\varepsilon_{n}:r_{n+1},\ldots,\varepsilon_2:r_1].
$$
The map $f_{q}$ acts as a one-sided shift on the Rosen expansion
of $x$: $f_q^n(x)=t_n$. We define the natural extension operator
to keep track of both $t_n$ and $v_n$.

\begin{Definition} \label{natural extension even} For a fixed $q$
the natural extension map $\mathcal{T}$is given by
\[ \mathcal{T}(x,y) = \left( f_q(x), \frac{1}{r\lambda + \varepsilon y} \right).\]
\end{Definition}
In \cite{[BKS]} it was shown that for every $q\geq 3$ there exists
a region $\Omega_q\subset \R^2$, for which $\mathcal{T}: \Omega_q
\to \Omega_q$ is bijective almost everywhere (with respect to an
invariant measure, see Equation ~(\ref{def: measure even}), that
is absolutely continuous with respect to Lebesgue measure). In
Section~\ref{sec: even-indices} (for $q$ even) and
Section~\ref{sec: odd-indices} (for $q$ odd) we recall the exact
form of $\Omega_q$. See also~\cite{[N1]}, where $\Omega_q$ was
obtained for $q$ even.

For $ x= [\varepsilon_1:r_1,\varepsilon_2:r_2,\ldots ]$ one has
$\mathcal{T}^n(x,0) = (t_n,v_n)$. The approximation coefficients
of $x$ can be given in terms of $t_n$ and $v_n$ (see also
\cite{[DK]}) as
\begin{equation} \label{eq: formulae for theta}
\Theta_{n-1}=\frac{v_n}{1+t_nv_n}, \quad \Theta_{n}=\frac{\varepsilon_{n+1}t_n}{1+t_nv_n}.
\end{equation}

For simplicity, we say that a real number $r/s$ is a {\em
$G_q$-rational} if it has finite Rosen expansion, all other real
numbers are called \emph{$G_q$-irrationals}. In~\cite{[HS]}, Haas
and Series derived a Hurwitz-type result using non-trivial
hyperbolic geometric techniques. They showed that for every
$G_q$-irrational $x$ there exist infinitely many $G_q$-rationals
$r/s$, such that $\Theta(x,r/s) \leq \mathcal{H}_q$, where
$\mathcal{H}_q$ is given by
$$
\mathcal{H}_q=\begin{cases}
\displaystyle{\frac{1}{2}} & \quad \textrm{if $q$ is even,}\\
\displaystyle{\frac{1}{2\sqrt{(1-\frac{\lambda}{2})^2+1}}} & \quad \textrm{if $q$ is odd.}\\
\end{cases}
$$
In this paper we derive a Borel-type  result, by showing that for
every $G_q$-irrational $x$ there are infinitely many $n\geq 1$
such that $\Theta_n \leq \mathcal{H}_q$. The even and odd case
differ and we treat them separately. In both cases we focus on
regions where
$\min\{\Theta_{n-1},\Theta_n,\ldots\}<\mathcal{H}_q$.

In fact,  the Borel-type result we derive does not immediately
imply the Hurwitz-type result of Haas and Series. Nakada
\cite{[N2]} showed that the Legendre constant $L_q$ is smaller
than $\mathcal{H}_q$ (recall that for the RCF this Legendre
constant is $1/2$, thus is larger than the Hurwitz constant
$1/\sqrt{5}$). Still, the Haas and Series results can be proved
using continued fraction properties by means of a map which yields
the Rosen-convergents and the so-called first medians;
see~\cite{[KNS]}.

\section{Tong's spectrum for even indices $q = 2p$}\label{sec: even-indices}
In this section $q$ is even, we fix $q = 2p$. The region of the
natural extension $\Omega_q$ is the smallest region where
$\mathcal{T}$ is bijective. Usually we write $\Omega$ instead of
$\Omega_q$. We have the following result from~\cite{[BKS]}.
\begin{Theorem}
\label{def: Omega even}\emph{\textbf{(\cite{[BKS]})}} The domain $\Omega$
upon which $\mathcal{T}$ is bijective is given by
 \[ \Omega = \bigcup_{j=1}^p J_j \times K_j. \]
Here $J_j$ is defined as follows: Let $\phi_j=T^j\left(
-\frac{\lambda}{2} \right)$, then $J_j=[\phi_{j-1},\phi_{j})$ for
$j \in \{1,2,\ldots,p-1 \}$ and $J_p = \left[ 0,\frac{\lambda}{2}
\right)$. Further, $K_j=[0,L_j]$ for $j \in \{1,2,\ldots,p-1 \}$
and $K_p = [0,R]$, where $L_j$ and $R$ are derived from the
relations
\begin{equation*}
\begin{cases}
({\mathcal R}_0): \qquad
R\, =\, \lambda - L_{p-1},\\
({\mathcal R}_1): \qquad
L_1\, =\, 1/(\lambda + R),\\
({\mathcal R}_j): \qquad L_j\, =\, 1/(\lambda -L_{j-1})\qquad \text{for }\,
j\in \{ 2,\cdots ,p-1 \},\\
({\mathcal R}_p): \qquad R\, =\, 1/(\lambda -L_{p-1}).
\end{cases}
\end{equation*}
\end{Theorem}

The map $f_q$ sends each interval $J_i$ to $J_{i+1}$ for
$i=1,\dots,p-1$. Further, we denote $\Omega_+ = \{(t,v) \in \Omega
\, |\, t>0\}$.

\begin{figure}[h!t]
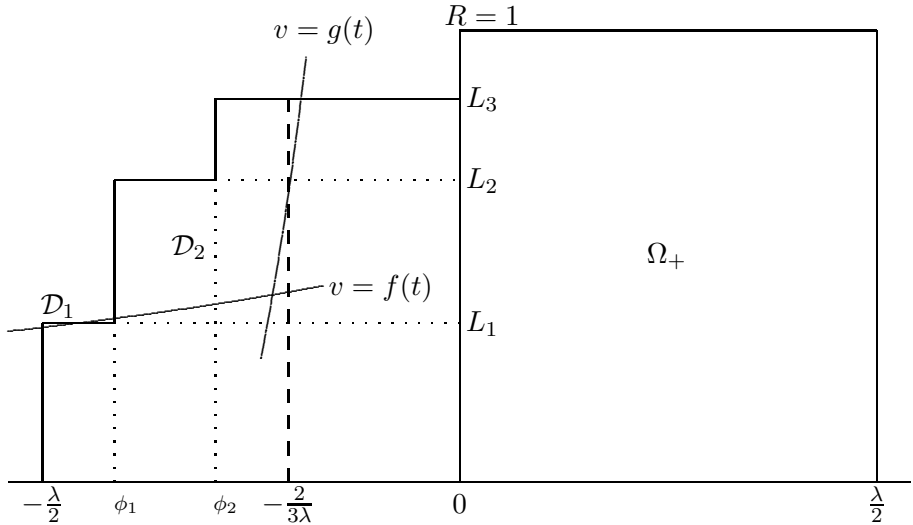

$$
\beginpicture
  \setcoordinatesystem units <0.6 cm, 0.6cm>
  \setplotarea x from -10 to 10, y from 0 to 10
  \putrule from -10 0 to 10 0
  \putrule from  0 0 to 0 10
  \putrule from  -9.24 0 to -9.24 3.51
  \putrule from  -9.24 3.51 to -7.65 3.51
  \putrule from  -7.65 3.51 to -7.65 6.68
  \putrule from  -7.65 6.68 to -5.41 6.68
  \putrule from  -5.41 6.68 to -5.41 8.48
  \putrule from  -5.41 8.48 to 0 8.48
  \putrule from  0 10 to 9.24 10
  \putrule from  9.24 0 to 9.24 10
  \put{$0$} at 0 -0.5
  \put{$-\frac{2}{3\lambda}$} at -3.8 -0.5
  \put{$\frac{\lambda}{2}$} at 9.24 -0.5
  \put{$-\frac{\lambda}{2}$} at -9.24 -0.5
  \put{\footnotesize $\phi_1$ \normalsize} at -7.3 -0.5
  \put{\footnotesize $\phi_2$ \normalsize} at -5.1 -0.5
  \put{$R=1$} at 0.5 10.35
  \put{$L_3$} at 0.5 8.48
  \put{$L_2$} at 0.5 6.68
  \put{$L_1$} at 0.5 3.51
   \put{$\Omega_+$ \large} at 4.7 5

\setquadratic \plot
 -10 3.333333333
-9.5    3.389830508
-9  3.448275862
-8.5    3.50877193
-8  3.571428571
-7.5    3.636363636
-7  3.703703704
-6.5    3.773584906
-6  3.846153846
-5.5    3.921568627
-5  4
-4.5    4.081632653
-4  4.166666667
-3.5    4.255319149
-3  4.347826087
/

\setquadratic \plot -4.4    2.727272727 -4.3    3.255813953 -4.2    3.80952381 -4.1    4.390243902
-4  5 -3.9    5.641025641 -3.8    6.315789474 -3.7    7.027027027 -3.6    7.777777778 -3.5
8.571428571 -3.4    9.411764706 /

\put{$v = f(t)$} at -1.75 4.3
\put{$v = g(t)$} at -3 10

\put{$\mathcal{D}_1$} at -8.9 3.9
\put{$\mathcal{D}_2$} at -6 5.2

  \setdots
  \putrule from  -7.65 0 to -7.65 3.51
  \putrule from  -5.41 0 to -5.41 6.68
  \putrule from  -7.65 3.51 to 0 3.51
  \putrule from  -5.41 6.68 to 0 6.68

   \setdashes
    \putrule from  -3.8 0 to -3.8 8.48

  \endpicture
$$
\caption[$\Omega$ for $q=8$]{\label{fig: Omegaq8} The region of
the natural extension $\Omega_8$, with $\mathcal{D}$ of Lemma
\ref{th: shape D}.}
\end{figure}

In~\cite{[BKS]} it is shown that $\mathcal{T}$ preserves a
probability measure,  $\nu$, that is absolutely continuous with
respect to Lebesgue measure. Its density is
\begin{equation}\label{def: measure even}
g_q(t,v)= \begin{cases}
\frac{\displaystyle C_q}{\displaystyle (1+tv)^2}, & \textrm{for } (t,v) \in \Omega_q,\\
0, & \textrm{otherwise},
\end{cases}
\end{equation}
where
$C_q=\displaystyle{\frac{1}{\log[(1+\cos\frac{\pi}{q})/\sin\frac{\pi}{q}]}}$
is a normalizing constant. It is also shown in \cite{[BKS]}, that
the dynamical system $(\Omega,\nu,\mathcal{T})$ is weak Bernoulli
(and therefore ergodic).

The following proposition on the distribution of the $\Theta_n$, also in \cite{[BKS]}, is
a consequence of the Ergodic Theorem and the strong approximation
properties of the Rosen fractions; see \cite{[DK]}, or
\cite{[IK]}, Chapter~4.

\begin{Proposition}
\label{prop: ergodic T even} Let $q \geq 3$ be even. For almost all
$G_q$-irrational numbers $x$ the two-dimensional sequence
\[\mathcal{T}^n(x,0) = \left( t_n,v_n \right), n \geq 1\]
is distributed over $\Omega_q$ according to the density function
$g_q(t,v)$ given in Equation ~{\rm (\ref{def: measure even})}.
\end{Proposition}

\subsection{Consecutive pairs of large approximation constants: The region $\mathcal{D}$}
To find Tong's spectrum we start by looking at two consecutive
large approximation coefficients $\Theta_{n-1}$ and $\Theta_n$. In
view of~(\ref{eq: formulae for theta}) we define
$\mathcal{D}\subset \Omega$ by
\begin{equation}
\label{eq: def D}
\mathcal{D}=\left\{(t,v) \in \Omega \,|\, \min\left\{ \frac{v}{1+tv},
\frac{|t|}{1+tv}\right\}>\frac12\right\}.
\end{equation}
So $(t_n,v_n)\in \mathcal{D}$ if and only if
$\min\{\Theta_{n-1},\Theta_n\}>\frac12$. We have the following
result describing $\mathcal{D}$.
\begin{Lemma}
\label{th: shape D}Define functions $f$ and $g$ by
\begin{equation} \label{eq: f and g}
f(x)=\frac{1}{2-x} \quad\textrm{and} \quad g(x)=\frac{2|x|-1}{x}.
\end{equation}
For all even $q,$ $\mathcal{D}$ consists of two connected
components $\mathcal{D}_1$ and $\mathcal{D}_2$.  The subregion
$\mathcal{D}_1$ is bounded by the lines $t = -\frac{\lambda}{2}, v
= L_1$ and the graph of $f$;  $\mathcal{D}_2$ is bounded by the
graph of $g$ from the right, by the graph of $f$ from below and by
the boundary of $\Omega$; see Figure~\ref{fig: Omegaq8}.
\end{Lemma}

\begin{proof}
For the approximation coefficients one has
\[
\begin{aligned}
\frac{v}{1+tv} \leq \frac12 &\Leftrightarrow& \quad v \leq f(t),\\
\frac{|t|}{1+tv} \leq \frac12 &\Leftrightarrow&
\begin{cases}
v \leq g(t) & \quad \textrm{if } t < 0\\
\\
v \geq g(t) &
\quad \textrm{if } t \geq 0.
\end{cases}
\end{aligned}
\]

Since for $t>0$ the graphs of $f$ and $g$ meet at $t=1$, and $1
> \frac{\lambda}{2}$, it follows that points for which
$\min\left\{\frac{v}{1+tv} ,\frac{|t|}{1+tv} \right\}<\frac12$
must satisfy $t<0$ and $v > g(t)$. It is easy to check that for
all even $q$ the only intersection points of the graphs of the
functions $f$ and $g$ in the region of the natural extension are
given by $\left(-\frac{\lambda}{2},\frac{2}{\lambda+4}\right)$,
$(-L_1,L_{p-1})$ and $(-L_{p-1},L_1)$. The fact that $\mathcal{D}$
consists of the two pieces follows from $\phi_0 \leq -L_{p-1} =
1-\lambda \leq \phi_1$.
\end{proof}

Having control on the approximation coefficients $\Theta_{n-1}$
and $\Theta_{n}$, we turn our attention to $\Theta_{n+1}$ on
$\mathcal{D}$. It follows from~(\ref{eq: formulae for theta}) and
the definition of $\mathcal{T}$ that

\begin{equation}
\label{eq: formula Theta_{n+1}}
\Theta_{n+1}=\frac{\varepsilon_{n+2}(1-\varepsilon_{n+1}r_{n+1}t_n\lambda)(\lambda r_{n+1}
+\varepsilon_{n+1}v_n)}{1+t_nv_n}.
\end{equation}

In order to express $\Theta_{n+1}$  locally as a function of only
$t_n$ and $v_n$,  we divide $\mathcal{D}$ into regions where
$r_{n+1},\varepsilon_{n+1}$ and $\varepsilon_{n+2}$ are constant.
This gives three regions; see Table~\ref{ConstantsEven} for the
definition of $\mathcal A$, $\mathcal B$, and $\mathcal C$, the
new subregions involved in this.   See also Figure~\ref{fig:Dq8}.

\begin{table}
\begin{displaymath}
\begin{array}{lrl|c|c|c|c}
\multicolumn{3}{c|}{\textrm{Region}} & \kern2pt r_{n+1} &\kern2pt  \varepsilon_{n+1}
&\kern2pt  \varepsilon_{n+2} &\kern2pt  \Theta_{n+1}\\
\hline
&&&&&\\
\mathcal A:& \displaystyle{\frac{-2}{3\lambda}} \leq t_n \leq &\!\!\!\! \frac{-1}{\lambda +1} &
2 & -1 & -1 &
\displaystyle{\frac{(1+2t_n\lambda)(v_n-2\lambda)}{1+t_nv_n}}\\
&&&&&\\
\mathcal B:& \displaystyle{\frac{-1}{\lambda}} \leq t_n < &\!\!\!\!\displaystyle{\frac{-2}{3\lambda}} & 1 & -1 & 1 &
\displaystyle{\frac{(1+t_n\lambda)(\lambda-v_n)}{1+t_nv_n}}\\
&&&&&\\
\mathcal C \cup \mathcal D_{1}:& \frac{-\lambda}{2} \leq t_n <&\!\!\!\! \displaystyle{\frac{-1}{\lambda}} & 1 & -1 & -1 &
\displaystyle{\frac{(1+t_n\lambda)(v_n-\lambda)}{1+t_nv_n}}\\
\end{array}
\end{displaymath}
\caption{Subregions of $\mathcal D$ giving constant coefficients.}
\label{ConstantsEven}
\end{table}

\begin{figure}
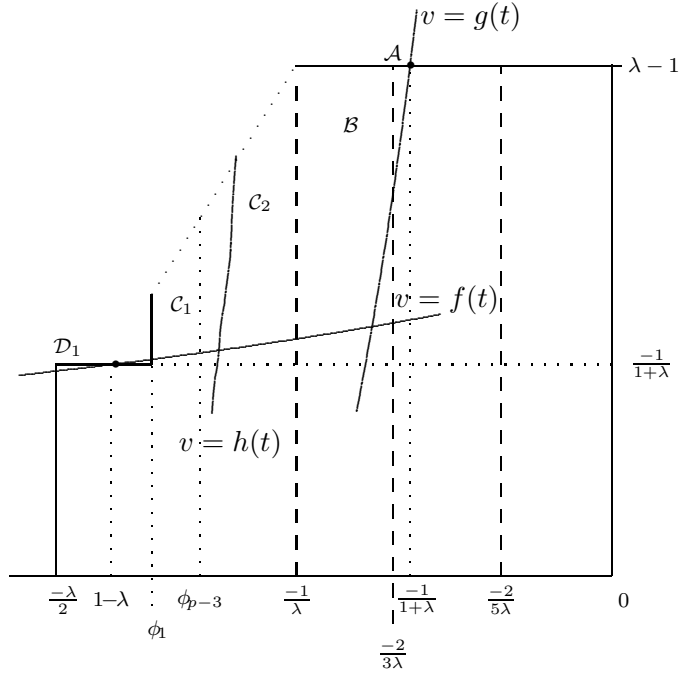

$$
\beginpicture
  \setcoordinatesystem units <0.8 cm, 0.8 cm>
  \setplotarea x from -10 to 0, y from 0 to 8.5
  \putrule from 0 0 to 0 8.5
  \putrule from -10  0 to 0 0
  \putrule from  -9.24 0 to -9.24 3.51
  \putrule from  -9.24 3.51 to -7.65 3.51
  \putrule from  -7.65 3.51 to -7.65 4.68
  \putrule from  -7.65 4.68 to -5.41 5.68
  \putrule from  -5.41 8.48 to 0 8.48

  \setquadratic \plot
-10 3.333333333 -9.5 3.389830508 -9
 3.448275862 -8.5    3.50877193 -8  3.571428571 -7.5
3.636363636 -7  3.703703704 -6.5    3.773584906 -6  3.846153846 -5.5    3.921568627 -5  4 -4.5
4.081632653 -4  4.166666667 -3.5    4.255319149 -3  4.347826087 /

\setquadratic \plot -4.4    2.727272727 -4.3    3.255813953 -4.2    3.80952381 -4.1    4.390243902
-4  5 -3.9    5.641025641 -3.8    6.315789474 -3.7    7.027027027 -3.6    7.777777778 -3.5
8.571428571 -3.4    9.411764706 /

\setquadratic \plot -6.8   2.7  -6.78  3   -6.75  3.3  -6.7   3.6 -6.675   3.9 -6.65   4.2 -6.6   4.5 -6.55  4.9 -6.5   5.3   -6.45   6.3   -6.4     7   /
\put{\footnotesize  $\mathcal D_{1}$} at -9.2 3.8
\put{\footnotesize  $\mathcal C_{1}$} at -7.3 4.5
\put{\footnotesize  $\mathcal C_{2}$} at -6 6.2
\put{\footnotesize $\mathcal B$} at  -4.5 7.5
\put{\footnotesize $\mathcal A$} at -3.8 8.7
\put{$v = f(t)$} at -2.9  4.6
\put{$v = g(t)$} at -2.5 9.3
\put{$v = h(t)$} at -6.5  2.2
\setdots
\setlinear \plot -7.8 4.68 -5.41 8.48 /
\put{\footnotesize$\frac{-\lambda}{2}$} at -9.28 -0.4

\putrule from -8.48 0 to  -8.48 3.51 \put{\footnotesize $1\!\!-\!\!\lambda$} at -8.48 -0.4

\putrule from  -7.8 -0.5 to -7.8 3.51 \put{\footnotesize $\phi_{\!1}$} at -7.7 -0.9

\putrule from  -7 0 to -7 6 \put{\footnotesize $\phi_{\!p-3}$} at -7 -0.4

\putrule from -3.511 0  to -3.511 8.48 \put{\footnotesize $\frac{-1}{1+\lambda}$} at -3.4 -0.4

\put{\tiny $\bullet$} at -3.5 8.48 \put{\tiny $\bullet$} at -8.4 3.511

\putrule from  -7.65 3.51 to 0 3.51 \put{\footnotesize $\frac{-1}{1+\lambda}$} at 0.5 3.51

\setdashes

\putrule from  -5.4 0 to -5.4 8.3 \put{\footnotesize $\frac{-1}{\lambda}$} at -5.4 -0.4

\putrule from -3.8 -0.8 to -3.8 8.48 \put{\footnotesize$\frac{-2}{3\lambda}$} at -3.8 -1.3

\putrule from -2 0 to -2 8.48 \put{\footnotesize$\frac{-2}{5\lambda}$} at -2 -0.4

\put {\footnotesize$0$ \normalsize} at 0.1 -0.4 \put {\footnotesize$\lambda-1$ \normalsize} at 0.6
8.48
\endpicture
$$
\caption[Area D q=even]{\label{fig:Dq8} The regions in $\mathcal{D}$.}
\end{figure}

Solving for $\Theta_{n+1}= 1/2$, leads to
$$
h(t) = \frac{2\lambda^2t+2\lambda+1}{2\lambda t - t +2}\,,
$$
whose graph divides  $\mathcal C$ into two parts. Let $\mathcal
C_{1}$ denote the left-hand side of this graph, there and on
$\mathcal D_{1}$ one has $\Theta_{n+1}>\frac12$; on the remainder,
$\mathcal C_{2}$,  one has  $\Theta_{n+1}<\frac12$.  Note that
$\mathcal T$ takes the graph of $h$ to the graph of $g$.

On its right-hand side region ${\mathcal B}$ is bounded by the
graph of $g$. In view of Equation ~(\ref  {eq: formula
Theta_{n+1}}) and Table~\ref {ConstantsEven}, we consider the
graph of
$$
\ell (t)=\frac{2\lambda^2t+2\lambda -1}{2\lambda t+t+2},\quad
\text{for $t\neq \frac{-2}{2\lambda +1}$.}
$$
An easy calculation shows that the graphs of $\ell$ and $g$
intersect only at the point $(-1/(\lambda +1),\lambda -1)$, and
that for $t>-2/(2\lambda +1)$ the graph of $g$ lies above that of
$\ell$. Furthermore, $\ell^{\prime}(t)>0$ for $t \neq
\frac{-2}{2\lambda +1}$, and $\ell \left(
\frac{-1}{\lambda}\right) =\lambda\,$; we conclude that
$\Theta_{n+1}<1/2$ on region ${\mathcal B}$.

\begin{Lemma}\label{lem: theta's on region (III)}
With notation as above, the subset of $\mathcal D$ on which
$\Theta_{n+1}> 1/2$ is exactly   the union of regions $\mathcal
D_{1}$, $\mathcal C_{1}$ and region $\mathcal A$. On region
$\mathcal A$ one has  $\Theta_{n+1}>\Theta_{n-1}>\Theta_{n}$.
\end{Lemma}
\begin{proof}
The remarks directly above show that we need now only consider region $\mathcal A$.

It immediately follows from~(\ref{eq: formulae for theta}) and the
fact that $v_n > -t_n$ on $\mathcal A$, that
$\Theta_{n-1}>\Theta_n$. To show that $\Theta_{n+1}>\Theta_{n-1}$
we need to show
\[(1+2t_n\lambda)(v_n-2\lambda)> v_n.\]
or equivalently
\[-4\lambda^2 t_n + 2\lambda t_nv_n -2\lambda>0.\]
We use $v_n > -t_n$ again, so it is enough to show
\[-4\lambda^2 t_n - 2\lambda t_n^2 -2\lambda\geq0.\]
The last statement is true if $t_n \in [-\lambda - \sqrt{\lambda^2-1}, -\lambda
+ \sqrt{\lambda^2-1}]$, which does indeed hold on region $\mathcal A$.

Since $\min\{\Theta_{n-1},\Theta_n\}>\frac12$ on $\mathcal{D}$,
the result follows.
\end{proof}

Now, by definition, $f_{q}(t) = -1/t- \lambda$   for $t \in [-\lambda/2, -2/3 \lambda)$. It follows that the $\mathcal T$ orbit
of any point  of the $\mathcal C_{i}$ either eventually leaves
$\mathcal D$, or eventually enters $\mathcal A$.   Thus,  we
naturally focus on the interval $ t \in J_{p-1} = [\phi_{p-2},0)$.
For almost every $G_q$-irrational $x\in
\left[-\frac{\lambda}{2},\frac{\lambda}{2}\right)$ there is an $n$
such that $\mathcal{T}^n(x,0)=(t_n,v_n)$ and $t_n \in J_{p-2}$.
Divide the interval $J_{p-2}$ into three parts. If $t_n \in
\left[\phi_{p-2},\frac{-2}{3\lambda}\right)$, then
$\min\{\Theta_{n-1},\Theta_{n},\Theta_{n+1}\}<\frac12$. If $t_n
\in \left[\frac{-1}{\lambda+1},0\right)$, then $(t_n,v_n) \notin
\mathcal{D}$ so $\min\{\Theta_{n-1},\Theta_{n}\}<\frac12$.
However, if $t_n \in
\left[\frac{-2}{3\lambda},\frac{-1}{\lambda+1}\right)$, then it
may be that $(t_n,v_n) \in \mathcal{A}$  and thus
$\min\{\Theta_{n-1},\Theta_{n},\Theta_{n+1}\}>\frac12$.

\begin{Lemma} The transformation $\mathcal{T}$ maps
$\mathcal{A}$  bijectively  onto region $\mathcal{D}_{1}$.
\end{Lemma}
\begin{proof}
The vertices of $\mathcal{A}$ are mapped onto the vertices of
region $\mathcal{D}_{1}$ by $\mathcal{T}$:
\begin{eqnarray*}
\left(\frac{-2}{3 \lambda},\lambda-1 \right)
& \mapsto& \left(\frac{-\lambda}{2}, \frac{1}{\lambda+1}\right), \\
\left(\frac{-1}{\lambda+1},\lambda-1 \right)
& \mapsto& \left(1-\lambda, \frac{1}{\lambda+1}\right), \\
\left(\frac{-2}{3 \lambda},\frac{3\lambda -4}{2} \right) &
\mapsto
& \left(\frac{-\lambda}{2}, \frac{2}{\lambda+4}\right).
\end{eqnarray*}

It is easily checked that  $\mathcal T$ takes the graph of $g$ to
the graph of $f$. Since $\mathcal{T}$ is continuous and bijective
and sends straight lines to straight lines, this completes the
proof.
\end{proof}

Under $\mathcal{T}$, region $\mathcal{D}_{1}$ is mapped onto a
region with upper vertices $(\phi_1,L_2)$ and $(-L_{p-2},L_2)$.
For $i=2,\dots,p-1$ we find  a region with upper vertices
$(\phi_{i-1},L_i)$ and $(-L_{p-i},L_i)$ after applying
$\mathcal{T}^{i}$  to $\mathcal{A}$.  The lower part of this
region is bounded by the $i$th transformation of the curve $g(t)$
under $\mathcal{T}$. After $p-1$ applications of $\mathcal{T}$
this results in a region with upper vertices
$(\phi_{p-2},L_{p-1}),(-L_{1},L_{p-1})$. Clearly, this region
intersects with $\mathcal{A}$. We call $p-1$ consecutive
applications of $\mathcal{T}$ a  {\em round}.

That part of $\mathcal T^{p-1}(\mathcal A)$ lying to the left of
$t=\frac{-2}{3\lambda}$ is in region $\mathcal{B}$; the images of
these points under a subsequent application of  $\mathcal T$ have
positive $t$-coordinate. We call \emph{flushing} an application of
$\mathcal T$ to such points --- the points are flushed from
$\mathcal{D}$. The remainder of $\mathcal{T}^{p-1}(\mathcal{A})$
is a subset of $\mathcal{A}$.

\begin{Theorem}
\label{th: flushing} Any point $(t,v)$ of $\mathcal{A}$ is flushed
after exactly $k$ rounds if and only if $$ \tau_{k-1} \leq t <
\tau_{k},
$$
where $\tau_{0}=\frac{-2}{3\lambda}$ and
$$
\tau_{k}=\left[\left(-1:2,\left(-1:1\right)^{p-2}\right)^k,
\left(\frac{-2}{3\lambda}:\right)\right].
$$
For any $x$ with $\tau_{k-1}\leq t_n < \tau_k$,
$$
\min\{\Theta_{n-1},\Theta_{n},\ldots,\Theta_{n+k(p-1)-1},
\Theta_{n+k(p-1)}\}>\frac12,
$$
while
$$
\min\{\Theta_{n-1},\Theta_{n},\ldots,\Theta_{n+k(p-1)},
\Theta_{n+k(p-1)+1}\}<\frac12.
$$
\end{Theorem}

\begin{proof}
A point $(t,v) \in \mathcal{A}$ gets flushed after exactly $k$
rounds if $k$ is minimal such that $\mathcal{T}^{k(p-1)}(t,v)$ has
its first coordinate smaller than $\frac{-2}{3\lambda}$. We look at
pre-images of $t =\frac{-2}{3\lambda}$ under $f_q$:
\begin{eqnarray*}
f_q^{-1}\left(\frac{-2}{3\lambda}\right)
& = & \left[-1;1,\left(\frac{-2}{3\lambda}:\right) \right],\\
f_q^{-(p-1)}\left(\frac{-2}{3\lambda}\right)
& = & \left[-1;2,(-1:1)^{p-2},\left(\frac{-2}{3\lambda}:\right) \right],\\
& \vdots&\\
f_q^{-k(p-1)}\left(\frac{-2}{3\lambda}\right)
& = & \left[\left(-1;2,(-1:1)^{p-2}\right)^k,\left(\frac{-2}{3\lambda}:\right) \right] = \tau_k.
\end{eqnarray*}
The result thus clearly follows.
\end{proof}

\subsection{Metrical results}
We define $\mathcal{A}_k= \left\{ (t,v) \in \mathcal{A}\, |\, t
\geq\tau_k  \right\}$ for $k\geq 0$. From Theorem~\ref{th:
flushing} and the ergodicity of $\mathcal{T}$ (cf.\
Proposition~\ref{prop: ergodic T even}) we have the following
metrical theorem on the distribution of large blocks of ``big"
approximation coefficients.
\begin{Theorem}
\label{th: limit even} For almost all $x$ (with respect to
Lebesgue measure) and $k \geq 1$, the limit
\begin{eqnarray*} && \lim_{N \rightarrow \infty}
\frac{1}{N} \; \# \left\{  1 \leq j \leq N \, | \,
\min \{\Theta_{j-1}, \Theta_j, \ldots, \Theta_{j+k(p-1)} \}
> \frac12 \right.\\
&& \;\;\;\;\;\;\;\;\;\;\;\;\;\;\;\;\;\;\;\;\;\; \left.
\textrm{and } \Theta_{j+k(p-1)+1} < \frac12 \right\}
\end{eqnarray*}
exists and equals $\nu(\mathcal{A}_{k-1}\setminus \mathcal{A}_{k})
= \nu(\mathcal{A}_{k-1})-\nu(\mathcal{A}_{k})$.
\end{Theorem}

In order to apply Theorem~\ref{th: limit even} we compute
$\nu(\mathcal{A})$ and $\nu(\mathcal{A}_k)$:
\begin{eqnarray}
\nu(\mathcal{A})&=&C_q \int_{t=\frac{-2}{3\lambda}}^{-L_1} \int_{v=g(t)}^{L_{p-1}}
\frac{1}{(1+tv)^2}\,dv \,dt\nonumber \\
&=& -C_q \int_{\frac{-2}{3\lambda}}^{-L_1}\left(  \frac{1}{2t^2}
+  \frac{1}{t} -\frac{L_{p-1}}{1+L_{p-1}t} \right)\,dt \nonumber \\
&=& C_q \left( \frac{-1}{2L_1} + \log\left|L_{p-1}+\frac{-1}{L_1}\right|
+ \frac{3\lambda}{4} - \log\left|L_{p-1}-\frac{3\lambda}{2}\right| \right) \label{eq:area A explicit}.
\end{eqnarray}
Using the normalizing constant $C_q$ from Definition~\ref{def:
measure even}, $L_1=\displaystyle{\frac{1}{\lambda+1}}$ and
$L_{p-1}=\lambda-1$, it follows from~(\ref{eq:area A explicit})
that
$$
\nu(\mathcal{A}) = \displaystyle{\frac{1}{\log[(1+\lambda/2)/\sin \pi/q]}} \left(
\frac{\lambda-2}{4} + \log \frac{4}{\lambda+2}\right) \label{eq:
measure A}.
$$
Similarly,
\[
\begin{aligned}
\nu(\mathcal{A}_k)&=&C_q \int_{t=\tau_k}^{-L_1}
\int_{v=g(t)}^{L_{p-1}}
\frac{1}{(1+tv)^2}\,dv \,dt \\
\\
&=& \label{eq: measure A_i} C_q\left[\,
\frac{(-\lambda-1)\tau_k-1}{2\tau_k}+\log\left|\frac{2\tau_k}{(\lambda-1)\tau_k+1}\right|\;\right]\,.
\end{aligned}
\]

\begin{Example} If $q=8$ we have $ \nu(\mathcal{A}) = 4.6 \cdot 10^{-4}, \nu(\mathcal{A}_1)
= 7.6 \cdot 10^{-7},$ and $\nu(\mathcal{A}_2) = 6.7 \cdot
10^{-10}$. So Theorem~\ref{th: limit even} yields that for almost
every $x$ about $4.6 \cdot 10^{-2}\,\%$ of the blocks of consecutive
approximation coefficients of length $6$ have the property that
\[\min
\{\Theta_{j-1}, \Theta_j, \ldots, \Theta_{j+3} \} > \frac12
 \quad \textrm{and } \Theta_{j+4} < \frac12,\]
while about $7.6 \cdot 10^{-5}\,\%$ of the blocks of length $9$ have
the property that
\[\min
\{\Theta_{j-1}, \Theta_j, \ldots, \Theta_{j+6} \} > \frac12
 \quad \textrm{and } \Theta_{j+7} < \frac12.\]
 \end{Example}

\subsection{Tong's spectrum for even $q$}
We are now ready to determine the Tong spectrum for these Rosen
continued fractions. Recall that Theorem~\ref{th: tong even}
states that the minimum of $1 + k(p-1)$ consecutive values of
$\Theta_{j}$, beginning with $\Theta_{n-1}$ is less than
$c_{k-1}:=\frac{-\tau_{k-1}}{1+(\lambda-1)\tau_{k-1}}$, where the
$\tau_{j}$ are given in Theorem~\ref{th: flushing}.\smallskip

\emph{Proof of Theorem \ref{th: tong even}.} We start with $k=1$.
Assume that $(t_n,v_n) \in \mathcal{D}$, otherwise we are done.
For a certain $0\leq i \leq p-2$ we have $(t_{n+i},v_{n+i})$ is
either in $\mathcal{A}$ or flushed to $\Omega_+$. In the latter
case $\Theta_{n+i}<\frac12$ and we are done. So assume that
$(t,v)=(t_{n+i},v_{n+i})\in \mathcal{A}$. It follows from
Lemma~\ref{lem: theta's on region (III)} that
$\min\{\Theta_{n-1},\Theta_{n},\Theta_{n+1}\}=\Theta_{n}$ on
$\mathcal{A}$.  We give an upper bound for this minimum, by
determining the maximum value of $\Theta_n$.

On $\mathcal{A }$ one has  $\Theta_n = \frac{-t_n}{1+t_nv_n}$, so
we must find the point $(t,v)\in\mathcal{A}$ where $m(t,v) :=
\frac{-t}{1+tv}$ attains its maximum. We have
$$
\frac{\partial m(t,v)}{\partial t} = \frac{-1}{(1+tv)^2}<0 \quad
\textrm{and} \quad \frac{\partial m(t,v)}{\partial v} =
\frac{t^2}{(1+tv)^2}>0.
$$
Thus $\Theta_n$ attains its maximum at the upper left corner
$\left(\frac{-2}{3\lambda},\lambda-1\right)$ of $\mathcal{A}$.
This maximum equals $\frac{2}{\lambda +2} = c_0.$

It is left to show that $c_0$ is the best possible constant. We
need to check that for the first $p-2$ points in either direction
on the orbit of $\left(\frac{-2}{3\lambda},\lambda-1\right)$ one
has $\min\{\Theta_{n-1},\Theta_{n}\} \geq c_0.$ These  points on
the orbit of $\left(\frac{-2}{3\lambda},\lambda-1\right)$ are
\begin{eqnarray*}
&&\left(\left[(-1:1)^{p-2},\left(\frac{-2}{3\lambda}:\right)\right], L_1\right),
\left(\left[(-1:1)^{p-1},\left(\frac{-2}{3\lambda}:\right)\right],
L_2\right), \ldots,\\
&&\left(\left[(-1:1),\left(\frac{-2}{3\lambda}:\right)\right],
L_{p-2}\right), \left(\frac{-2}{3\lambda},\lambda-1\right),
\left(\frac{-\lambda}{2},L_1\right),\\&& \left(\phi_1,L_2\right),
\left(\phi_2,L_3\right),\ldots, (\phi_{p-3},L_{p-2}).
\end{eqnarray*}
We consider the curves $\Theta_{n-1} = c_0$  and $ \Theta_{n} =
c_0$, thus
$$
f_1(x) = \frac{2}{\lambda+2-2x} \quad \textrm{ and} \quad
g_1(x)=\frac{-(\lambda+2)x-2}{2x}.
$$
The graph of $f_1(x)$ intersects with the $x$-axis at the point
$\left(0,\frac{2}{\lambda+2}\right)$. For the heights $L_i$ we
easily find that $L_{p-2}>\ldots>L_3>L_{2}>\frac{2}{\lambda+2}.$
Thus, each of the first $p-2$ points in either direction of the
orbit has either $y$-coordinate greater than $L_2$ or is one of
the points $\left(-\frac{\lambda}{2},L_1\right)$ or
$\left(\left[(-1:1)^{p-2},\left(\frac{-2}{3\lambda}:\right)\right],L_1\right)$.
At the point $\left(-\frac{\lambda}{2},L_1\right)$ the value of
$\Theta_{n-1}$ is exactly $c_1 = \frac{2}{\lambda+2}$. Further, we
know that $\frac{\partial \Theta_{n-1}}{\partial t_n}>0$ on region
$\mathcal C$, so
\[\Theta_{n-1} \left(\left[(-1:1)^{p-2},
\left(\frac{-2}{3\lambda}:\right)\right],L_1\right)>
\Theta_{n-1}\left(-\frac{\lambda}{2},L_1\right)\,.\]

Together this means that for all these points one finds
$\Theta_{n-1}\geq c_0.$ The graph of $g_1(x)$ intersects with the
$x$-axis in the point $\left(\frac{-2}{\lambda+2},0\right)$. We
have
$$
-\frac{\lambda}{2}< \phi_{1} < \ldots < \phi_{p-1}<
\left[-1:1,\left(\frac{-2}{3\lambda}:\right)\right]<\frac{-2}{\lambda+2}\;,
$$
so for the relevant points on the orbit of
$\left(\frac{-2}{3\lambda}, \lambda-1\right)$ we have
$\Theta_{n}>c_0\,$.

This proves the case $k=1$. For general $k$, assume that the
starting point $(t_n,v_n) \in \mathcal{A}$ did not get flushed
during the first $k-1$ rounds, otherwise we are done. Consider
again a point $(t,v)=(t_{n+i},v_{n+i})$ in $\mathcal{A}$ with
$0\leq i \leq p-2$. From Theorem~\ref{th: flushing}  this means
that $t \geq \tau_{k-1}$. We find that the maximum of $\Theta_n$ occurs at the
upper-left corner of the region where $(t,v)$ is located. This
maximum is given by
\[\Theta_n(\tau_{k-1},\lambda-1)=\frac{-\tau_{k-1}}{1+(\lambda-1)\tau_{k-1}}
= c_{k-1}.\] The proof that this is the best possible is similar
to the case $k=1$. \qed

\bigskip

Note that
\[\lim_{k\rightarrow \infty} \tau_k = [\overline{(-1:2,(-1:1)^{p-2})}] = \frac{-1}{\lambda+1},\]
yielding
\begin{equation}\label{eq:borel-EvenCase}
\lim_{k\rightarrow \infty}
\frac{-\tau_k}{1+(\lambda-1)\tau_k}=\frac12.
\end{equation}

\begin{Lemma}\label{lem:FixedAndFlushedEven}   Let $\mathcal F$ denote the
fixed point set in ${\mathcal D}$ of ${\mathcal T}^{(p-1)}$. Then
\begin{enumerate}
\item[($i$)]  $\mathcal F = \{\, {\mathcal
T}^i(-1/(\lambda +1),\lambda -1)\,|\, i=0,1,\dots,p-1\}$;
\item[($ii$)]     For every  $x$ and every $n\geq 0$,
$(t_n,v_n) \notin \mathcal F$;
\item[($iii$)]  For every $G_q$-irrational number $x$ there
are infinitely many $n$ for which $(t_n,v_n)\notin {\mathcal D}$;
\item[($iv$)]  For each
$i=0,1,\dots,p-1$, let  $x_i=f_q^i(-1/(\lambda +1))\,$.   Then
for all $n\geq 0$,  ${\mathcal T}^n(x_i,0)\notin {\mathcal D}$.
However, ${\mathcal T}^{k(p-1)}(x_i,0)$ converges from below along
the vertical line $x=x_i$ to ${\mathcal T}^i(-1/(\lambda +1),
\lambda -1)$.
\end{enumerate}
\end{Lemma}

\begin{proof}   (i) Since ${\mathcal T}^{(p-1)}$ fixes $(-1/(\lambda +1),\lambda -1)$,
$\mathcal F$ certainly contains the $\mathcal T$-orbit of
$(-1/(\lambda +1),\lambda -1)$. On the other hand, any point of
$\mathcal D$ not in this orbit is eventually flushed from
$\mathcal D$.

(ii)  The second coordinate of any element of $\mathcal F$ is
$G_q$-irrational, but for all $x$, the points $(t_n,v_n)$ have
$G_q$-rational second coordinate.

(iii) From the previous item, for any $G_q$-irrational $x$  the
$\mathcal T$-orbit of $(x,0)$  avoids $\mathcal F$, hence each
time this orbit encounters $\mathcal D$, it is flushed out within
a finite number of iterations.   We conclude indeed that for every
$G_q$-irrational $x$ there are infinitely many $n$ for which
$(t_n,v_n)\notin {\mathcal D}$.

(iv) The final item is easily checked.
\end{proof}

Combining Theorem~\ref{th: tong even},
Equation~(\ref{eq:borel-EvenCase}) and the previous lemma, we have
the following result.
\begin{Theorem}\label{thm:BorelEven}
For every $G_q$-irrational $x$ there are infinitely many $n \in
\N$ for which
\[ \Theta_n \leq \mathcal{H}_q = \frac12.\]
The constant $1/2$ is best possible.
\end{Theorem}

\section{Tong's spectrum for odd indices $q=2h+3$}
\label{sec: odd-indices}

The classical case of $q=3$ is Tong's result itself; a geometric
argument is given in \cite{[HK]}.    In our treatment, the case of
$q=5$  is also exceptional and we do not give full details for it.
See in particular the remarks  after Lemmata ~\ref{lem: f g odd} and~\ref{lem: T A D}.

The results for the odd case are derived similarly to those of the
even case. However, this case has more complicated dynamics,
complicating the arguments.     We begin with the
definition of the natural extension.

Again let $\phi_j = f_q^j(-\lambda/2)$.  Set $h=\frac{q-3}{2}$
and define $J_j,\, j\in  \{ 1,\cdots , 2h+2 \}$, by
\begin{align*}
J_{2k}& = [\phi_{h+k},\phi_k), {\mbox{ for }} k\in \{ 1,\cdots , h \} ,\\
J_{2k+1}& = [\phi_k,\phi_{h+k+1}), {\mbox{ for }} k\in \{ 0,1,\cdots , h \} ,
\end{align*}
and $J_{2h+2}\, =\, [0,\frac{\lambda}{2})$. Let $K_j=[0,L_j]$ for
$j\in \{ 1,\cdots , 2h+1 \} $ and let $K_{2h+2} =[0,R]$, where $R$
is the solution of
\[R^2 + (2-\lambda) R-1 =0.\]

\begin{Theorem}\emph{\textbf(\cite{[BKS]})}\label{thm2} Let $q = 2 h + 3$, with $ h \ge 1$
and $\Omega\, =\, \bigcup_{j=1}^{2h+2} J_j\times K_j$. With $R$
defined as above, we have
\begin{equation*}
\left\{
\begin{array}{llll}
({\mathcal R}_0): & R &=&\lambda - L_{2h+1},\\
({\mathcal R}_1): & L_1& =&1/(2\lambda -L_{2h}),\\
({\mathcal R}_2): & L_2 &=&1/(2\lambda -L_{2h+1}),\\
({\mathcal R}_j): & L_j &=& 1/(\lambda -L_{j-2}) \qquad \text{for }\, 2< j < 2h+2,\\
({\mathcal R}_{2h+2}): & R& =& 1/(\lambda -L_{2h}),
\end{array}
\right.
\end{equation*}
while ${\mathcal T}:\, \Omega \rightarrow \Omega$ is bijective off
a set of Lebesgue measure zero.
\end{Theorem}
\begin{figure}
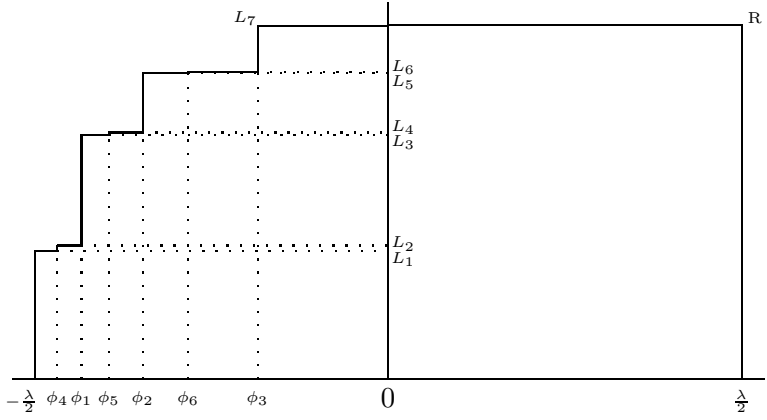

$$
\beginpicture
  \setcoordinatesystem units <0.5 cm, 0.5cm>
  \setplotarea x from -10 to 10, y from 0 to 10
  \putrule from -10 0 to 10 0
  \putrule from  0 0 to 0 10
  \putrule from  -9.397 0 to -9.397 3.4
  \putrule from  -9.397 3.4 to -8.794 3.4
  \putrule from  -8.794 3.4 to -8.794 3.545
  \putrule from  -8.794 3.545 to -8.152 3.545
  \putrule from  -8.152 3.545 to -8.152 6.497
  \putrule from  -8.152 6.497 to -7.422 6.497
  \putrule from  -7.422 6.497 to -7.422 6.556
  \putrule from  -7.422 6.556 to -6.527 6.556
  \putrule from  -6.527 6.556 to -6.527 8.131
  \putrule from  -6.527 8.131 to  -5.321 8.131
  \putrule from   -5.321 8.131 to -5.321 8.173
  \putrule from   -5.321 8.173 to -3.473 8.173
  \putrule from   -3.473 8.173 to -3.473 9.379
  \putrule from   -3.473 9.379 to 0 9.379
  \putrule from   0 9.415 to 9.397 9.415
  \putrule from   9.397 0 to 9.397 9.415
  \put{$0$} at 0 -0.5
\tiny
  \put{$\frac{\lambda}{2}$} at 9.397 -0.6
  \put{$-\frac{\lambda}{2}$} at -9.75 -0.6
  \put{$\phi_1$} at -8.152 -0.5
  \put{$\phi_2$} at -6.527 -0.5
  \put{$\phi_3$} at -3.473 -0.5
  \put{$\phi_4$} at -8.794 -0.5
  \put{$\phi_5$} at -7.422 -0.5
  \put{$\phi_6$} at -5.321 -0.5
  \put{$L_1$} at  0.4 3.2
  \put{$L_2$} at  0.4 3.6
  \put{$L_3$} at  0.4 6.3
  \put{$L_4$} at  0.4 6.7
  \put{$L_5$} at  0.4 7.9
  \put{$L_6$} at  0.4 8.3
  \put{$L_7$} at  -3.8 9.6
  \put{R} at 9.75 9.6
\normalsize
\setdots
    \putrule from  -8.152 0 to -8.152 3.545
    \putrule from  -6.527 0 to -6.527 6.556
    \putrule from  -3.473 0 to -3.473 8.173
    \putrule from  -8.794 0 to -8.794 3.4
    \putrule from  -7.422 0 to -7.422 6.496
    \putrule from  -5.321 0 to -5.321 8.131
    \putrule from  -8.152 3.545 to 0 3.545
    \putrule from  -6.527 6.556 to 0 6.556
    \putrule from  -3.473 8.173 to 0 8.173
    \putrule from  -8.794 3.4 to 0 3.4
    \putrule from  -7.422 6.496 to 0 6.496
    \putrule from  -5.321 8.131 to 0 8.131
    \endpicture
$$
\caption[$\Omega$ for $q=9$]{\label{fig: Omegaq9}The region of the
natural extension $\Omega_9$.}
\end{figure}

In~\cite{[BKS]} it is shown that $\mathcal{T}$ preserves the
probability measure $\nu$, with density
$\displaystyle{\frac{C_q}{(1+xy)^2}},$ where
$C_q~=~\displaystyle{\frac{1}{\log(1+R)}}$ is a normalizing
constant. It is also shown in~\cite{[BKS]} that the process
$(\Omega,\nu,\mathcal{T})$ is weak Bernoulli and therefore
ergodic. Proposition~\ref{prop: ergodic T even} also holds in the
odd case.

For odd indices $q$ we define again the region $\mathcal{D}$ by
$$
\mathcal{D}=\left\{(t,v) \in \Omega |
\min\left\{\frac{v}{1+tv},\frac{|t|}{1+tv}\right\}>\mathcal{H}_q
\right\} .
$$
Clearly $\min\{\Theta_{n-1},\Theta_n\}> \mathcal{H}_q$ if and only if $(t_n,v_n) \in \mathcal{D}$.
\begin{Lemma} Define functions $f$ and $g$ by
\label{lem: f g odd}
\begin{equation} \label{eq: f and g odd}
f(x)=\frac{\mathcal{H}_q}{1-\mathcal{H}_qx} \quad\textrm{and } g(x)=\frac{|x|-\mathcal{H}_q}{\mathcal{H}_qx}.
\end{equation}
For all odd $q\geq 7$, the region $\mathcal{D}$ consists of four disjoint regions:
\begin{itemize}
\item [$\mathcal{D}_1$]  bounded by the lines $x=\frac{-\lambda}{2}$ and $y=L_1$ and the graph of $f$;
\item [$\mathcal{D}_2$] bounded by the line $x=\phi_{h+1}$ and $y=L_2$ and the graph of $f$;
\item[$\mathcal{D}_3$]  (the largest region)  bounded from below by the graph of
$f$, from the right by the graph of $g$ and by the boundary of
$\Omega$; and,
\item[$\mathcal{D}_4$] bounded by the lines $x=\phi_h$ and $y = L_{2h+1}$ and the graph of $g$.
\end{itemize}
\end{Lemma}

\begin{Remark}\label{rmkFiveConn} In the exceptional case where $q = 5$, one
finds that $\mathcal D$ consists of a single connected component.
\end{Remark}

\begin{figure}
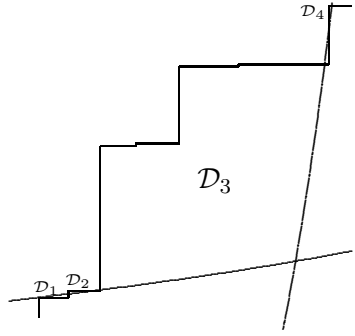

$$
\beginpicture
  \setcoordinatesystem units <0.65 cm, 0.65 cm>

  \putrule from  -9.397 3 to -9.397 3.4
  \putrule from  -9.397 3.4 to -8.794 3.4
  \putrule from  -8.794 3.4 to -8.794 3.545
  \putrule from  -8.794 3.545 to -8.152 3.545
  \putrule from  -8.152 3.545 to -8.152 6.497
  \putrule from  -8.152 6.497 to -7.422 6.497
  \putrule from  -7.422 6.497 to -7.422 6.556
  \putrule from  -7.422 6.556 to -6.527 6.556
  \putrule from  -6.527 6.556 to -6.527 8.131
  \putrule from  -6.527 8.131 to  -5.321 8.131
  \putrule from   -5.321 8.131 to -5.321 8.173
  \putrule from   -5.321 8.173 to -3.473 8.173
  \putrule from   -3.473 8.173 to -3.473 9.379
  \putrule from   -3.473 9.379 to -3 9.379

  \setquadratic \plot
-10 3.337368311
-9.5    3.39400351
-9  3.452594092
-8.5    3.513243104
-8  3.576060962
-7.5    3.641166122
-7  3.708685828
-6.5    3.778756939
-6  3.851526858
-5.5    3.927154565
-5  4.005811775
-4.5    4.087684243
-4  4.172973228
-3.5    4.261897141
-3  4.354693418
/

\setquadratic \plot -4.4    2.76354362
-4.3    3.292084846
-4.2    3.845794702
-4.1    4.426514795
-4  5.036270892
-3.9    5.677296533
-3.8    6.352060366
-3.7    7.063297919
-3.6    7.81404867
-3.5    8.607699464
-3.4    9.448035598
 /

\put{\tiny $\mathcal{D}_1$} at -9.25 3.6
\put{\tiny $\mathcal{D}_2$} at -8.6 3.75
\put{$\mathcal{D}_3$} at -5.8 5.8
\put{\tiny $\mathcal{D}_4$} at -3.8 9.25
\endpicture
$$
\caption[$\mathcal{D}$ for $q=9$]{\label{fig: Dq9} The region
$\mathcal{D}$ for $q=9$.}
\end{figure}

We begin our study of $\Theta_{n+1}$ by decomposing $\mathcal{D}$
into regions where  $r_{n+1},\varepsilon_{n+1}$ and
$\varepsilon_{n+2}$ are constant; see Table
~\ref{constantSubregionsOddCase} for the definition of the new
subregions involved in this.  See also Figure~\ref{fig:
regionsq9}.    Again, due to the dynamics of the situation, we
will need only focus on the $\mathcal T$-orbit of $\mathcal A$.

\begin{table}
\begin{displaymath}
\begin{array}{lrl|c|c|c|c}
\multicolumn{3}{c|}{\textrm{Region}} & \kern1pt r_{n+1} &\kern1pt  \varepsilon_{n+1} &\kern1pt \varepsilon_{n+2} &\Theta_{n+1}\\
\hline
&&&&&\\
\mathcal A \cup \mathcal D_{4}:&\displaystyle{\frac{-2}{3\lambda}} \leq t_n \leq &\!\!\!\!\frac{-1}{\lambda
+\frac{1}{R}} & 2 & -1 & -1 &
\kern-3pt \displaystyle{\frac{(1\!+\!2t_n\lambda)(v_n\!-\!2\lambda)}{1+t_nv_n}}\\
&&&&&\\
\mathcal B:& \displaystyle{\frac{-1}{\lambda}} \leq t_n < &\!\!\!\! \displaystyle{\frac{-2}{3\lambda}} & 1 & -1 & 1 &
 \displaystyle{\frac{(1+t_n\lambda)(\lambda-v_n)}{1+t_nv_n}}\\
&&&&&\\
\mathcal C \cup \mathcal D_{1} \cup \mathcal D_{2}:& \frac{-\lambda}{2} \leq t_n <&\!\!\!\!\displaystyle{\frac{-1}{\lambda}} & 1 & -1 & -1 &
\displaystyle{\frac{(1+t_n\lambda)(v_n-\lambda)}{1+t_nv_n}}\\
\end{array}
\end{displaymath}
\caption{Regions of constant coefficients, the odd case.}
\label{constantSubregionsOddCase}
\end{table}

\begin{figure}[h!t]
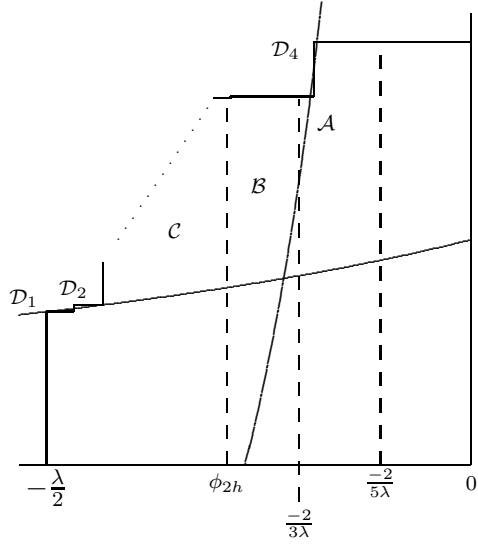

$$
\beginpicture
  \setcoordinatesystem units <0.6 cm, 0.6cm>
  \setplotarea x from -10 to 1, y from 0 to 10
  \putrule from -10 0 to 0 0
  \putrule from  0 0 to 0 10
  \putrule from  -9.397 0 to -9.397 3.4
  \putrule from  -9.397 3.4 to -8.794 3.4
  \putrule from  -8.794 3.4 to -8.794 3.545
  \putrule from  -8.794 3.545 to -8.152 3.545
  \putrule from  -8.152 3.545 to -8.152 4.497
  \putrule from  -5.7 8.131 to  -5.321 8.131
  \putrule from   -5.321 8.131 to -5.321 8.173
  \putrule from   -5.321 8.173 to -3.473 8.173
  \putrule from   -3.473 8.173 to -3.473 9.379
  \putrule from   -3.473 9.379 to 0 9.379
  \put{$-\frac{\lambda}{2}$} at -9.397 -0.5

\setquadratic \plot -10 3,329300796 -9  3,443960604 -8  3,566799784 -7  3,698725932 -6
3,840786083 -5  3,994194552 -4  4,160367733 -3  4,340967963 -2
4,537959324 -1  4,753679359 0 4,990932277
/

 \setquadratic \plot -4.9
0.371826474 -4.7    1.240258954 -4.5    2.185885431 -4.3
3.219477163 -4.1    4.353907112 -3.9    5.60468885 -3.7
6.990690236 -3.5    8.535091781 -3.3    10.26669351 /
\setlinear \plot -5 0
 -4.9   0.408
/

\put{\footnotesize $\mathcal D_{1}$} at -9.9 3.7
\put{\footnotesize $\mathcal D_{2}$} at -8.8 3.9
\put{\footnotesize$\mathcal C$} at -6.55 5.2
\put{\footnotesize $\mathcal B$} at -4.7 6.2
\put{\footnotesize $\mathcal A$} at -3.2 7.6
\put{\footnotesize $\mathcal D_{4}$} at -4.1 9.2
\setdots

\setlinear \plot -7.8 5.0 -5.7 8.1 /

\setdashes
\putrule from  -5.4 0 to -5.4 8.2
\put{\footnotesize
$\phi_{2h}$} at -5.4 -0.4

\putrule from -3.8 -0.8 to -3.8 8.1
\put{\footnotesize$\frac{-2}{3\lambda}$} at -3.8 -1.3

\putrule from -2 0 to -2 9.4
\put{\footnotesize$\frac{-2}{5\lambda}$} at -2 -0.4

\put {\footnotesize$0$ \normalsize} at 0.1 -0.4

\endpicture
$$
\caption[Regions in $\mathcal{D}$ for $q=9$]{\label{fig:
regionsq9} The regions for $\Theta_{n+1}$ on $\mathcal{D}$.}
\end{figure}

On  $\mathcal A \cup \mathcal D_{4}$,  one easily finds that
$\Theta_{n+1}>\Theta_{n-1}>\Theta_{n}$  holds.   Again flushing occurs to the left of the line
$t=\frac{-2}{3\lambda}$; note that $\phi_{2h}\leq \frac{-2}{3\lambda}
<\phi_{h}$. Thus we study  the interval $\left[\phi_{2h},
\phi_{h}\right)$ (instead of the interval $\left[\phi_{h},
0\right)$, as in the even case).    Note that the only one of our regions in the strip
defined by the interval $\phi_{2h}\leq t\leq \phi_{h}$ is indeed $\mathcal A$.

One easily shows the following result.
\begin{Lemma}
\label{lem: T A D}
The transformation $\mathcal{T}$ maps
$\mathcal{A}$ bijectively to region $\mathcal D_{1}$.
\end{Lemma}

\begin{Remark}\label{rmkFiveImA}   In the exceptional case where $q = 5$,
one can easily see that $\mathcal A$ can be defined as in general,
and has a similar image. The rest of our arguments can be
appropriately adjusted so that the results announced below  go
through for $q=5$.
\end{Remark}

The transformations of $\mathcal{A}$ follow a more complex orbit
than in the even case: here, the regions make a ``double loop,''
related to the orbit of the $\phi_{j}$. Region $\mathcal D_{1}$
gets transformed into a region with upper right vertex
$(-L_{2h-1},L_3)$, this region gets transformed in a region with
upper right vertex $(-L_{2h-3},L_5)$ and so on until we reach a
region with upper right corner $(-L_1,L_{2h+1})$, which lies in
region $\mathcal D_{4}$. This region gets transformed into a
region with upper right vertex $(-L_{2h},L_2)$.  Thereafter we
find a region with upper right corner $(-L_{2h-2},L_4)$ and so on,
until finally the image intersects with $\mathcal{A}$ after $2h+1$
steps. Here, we call a \emph{round} these $2h+1$ transformations
of $\mathcal{T}$.
\begin{Theorem}
\label{th: flushing odd}  Let the constants $\tau_{k}$ be defined
by
$$
\tau_k =
\left[\left(-1:2,(-1:1)^{h},-1:2,\left(-1:1\right)^{h-1}\right)^k,
\left(\frac{-2}{3\lambda}:\right)\right] .
$$
Then for any $(t,v)\in \mathcal{A}$ such that $\tau_{k-1} \leq t <
\tau_k$,
 the point $(t_n,v_n)$ is flushed after $k$ rounds.
In particular, for any $x\in \mathbb R$ with $\mathcal T^{n}(x,0)
= (t_{n}, v_{n})$ satisfying $\tau_{k-1} \leq t_n < \tau_k$ one
has
$$
\min\{\Theta_{n-1},\Theta_{n},\ldots,\Theta_{n+k(2h+1)},\Theta_{n+k(2h+1)}\}>
\mathcal{H}_q,
$$
while
$$
\min\{\Theta_{n-1},\Theta_{n},\ldots,\Theta_{n+k(2h+1)},\Theta_{n+k(2h+1)+1}\}<
\mathcal{H}_q.
$$
\end{Theorem}

\subsection{Metric results}
As in the even case, we define $\mathcal{A}_k= \left\{ (t,v) \in
\mathcal{A}\, |\, t>\tau_k  \right\}$ for $k\geq 0$.
\begin{Theorem}
\label{th: limit odd} For almost all $x$ (with respect to the Lebesgue measure) and $k \geq 1$ the limit
\begin{eqnarray*} && \lim_{N \rightarrow \infty} \frac{1}{N} \# \left\{  1 \leq j \leq N \, | \, \min
\{\Theta_{j-1}, \Theta_j, \ldots, \Theta_{j+k(2h+1)} \} > \frac12 \right.\\
&& \left. {\phantom{XXXXXXXXXXXX}}\textrm{and }
\Theta_{j+k(2h+1)+1} < \frac12 \right\}
\end{eqnarray*}
exists and equals $\nu(\mathcal{A}_{k-1}\setminus \mathcal{A}_{k}) = \nu(\mathcal{A}_{k-1})-\nu(\mathcal{A}_{k})$.
\end{Theorem}

We compute again $\nu(\mathcal{A})$ and $\nu(\mathcal{A}_k)$. A
calculation similar to that in the even case yields
\begin{eqnarray*}
\nu(\mathcal{A}) &=& D\left(\log \left(R+\frac{1}{R}\right) -
\log\left(\frac{\lambda R+2}{2R}\right) +\frac{\lambda
C-2CR}{2}\right) \label{eq: measure A1},\\
\nu(\mathcal{A}_k) &=& D\left(\log \left(R+\frac{1}{R}\right) -
\log\left|\frac{1+\tau_k(\lambda -1/R)}{\tau_k}\right|
-C(\lambda+R)-\frac{C}{\tau_k}\right) \nonumber \label{eq:
measure Ak odd}.
\end{eqnarray*}

\begin{Example} If $q=9$ we have $\nu(\mathcal{A}) = 6.2 \cdot 10^{-7},
 \, \nu(\mathcal{A}_1) = 6.5 \cdot 10^{-13},$ and $\nu(\mathcal{A}_2) = 6.8 \cdot 10^{-19}$.

So we find that for almost every $x$ about $6.5\cdot 10^{-11}\,\%$ of the blocks of
consecutive approximation coefficients of length $10$ have the
property that
\[\min \{\Theta_{j-1}, \Theta_j, \ldots, \Theta_{j+7} \} > \frac12
 \quad \textrm{and } \Theta_{j+8} < \frac12,\]
while about $6.8 \cdot 10^{-17}\,\%$ of the blocks of length $17$ have
the property that $$ \min \{\Theta_{j-1}, \Theta_j, \ldots,
\Theta_{j+14} \} > \frac12 \quad \textrm{and } \Theta_{j+15} <
\frac12.
$$
\end{Example}

\subsection{Tong's spectrum for odd $q$}
We have the following result, which is proved similarly to Theorem~\ref{th: tong even}.

\begin{Theorem}
For every $G_q$-irrational number $x$ and all positive $n$ and $k$,  one has
\[ \min \{ \Theta_{n-1},\Theta_n,\ldots,\Theta_{n+k(2h+1)}\} < \frac{-\tau_{k-1}}{1+(\lambda-R)\tau_{k-1}}.\]
The constant $c_{k-1}=\frac{-\tau_{k-1}}{1+(\lambda-R)\tau_{k-1}}$ is best possible.
\end{Theorem}

Note that $\lim_{k\rightarrow \infty}c_k = \mathcal{H}_q$. Due to
this, and reasoning as in the proof of
Lemma~\ref{lem:FixedAndFlushedEven} (here one considers the
fixed-points in ${\mathcal D}$ of ${\mathcal T}^{2h+1}$ instead of
those of ${\mathcal T}^{p-1}$), we get the following result.
\begin{Theorem}
For every $G_q$-irrational $x$ there are infinitely many $n \in \N$, such that
\[\Theta_n(x) \leq \mathcal{H}_q.\]
The constant $\mathcal{H}_q$ is best possible.
\end{Theorem}

\begin{acknowledgements}
We   thank the referee for a careful reading of this paper.
\end{acknowledgements}


\begin{thebibliography}{[BJW]}
\bibitem[A]{[A]} Adams, William W. -- \emph{On a relationship between the convergents
of the nearest integer and regular continued fractions}, Math.\
Comp.\ {\bf 33} (1979), no. 148, 1321--1331.

\bibitem[B]{[B]} Burger, E.B. -- \emph{Exploring the number jungle: a
journey into Diophantine analysis}, Student Mathematical Library,
{\bf 8}. American Mathematical Society, Providence, RI, 2000.

\bibitem[BJW]{[BJW]} Bosma, W., Jager, H.\ and Wiedijk, F. ---
\emph{Some metrical observations on the approximation by continued
fractions}, Nederl.\ Akad.\ Wetensch.\ Indag.\ Math.\ {\bf 45}
(1983), no. 3, 281--299.

\bibitem[BKS]{[BKS]} Burton, R.M., Kraaikamp, C. and Schmidt, T.A. --
\emph{Natural extensions for the Rosen fractions}, Trans.\ Amer.\
Math.\ Soc.\ {\bf 352} (1999), 1277--1298.

\bibitem[CF]{[CF]} Cusick, T.W. and Flahive, M.E. -- \emph{The
Markoff and Lagrange spectra}, Mathematical Surveys and
Monographs, {\bf 30}. American Mathematical Society, Providence,
RI, 1989.

\bibitem[DK]{[DK]} Dajani, K. and Kraaikamp, C. -- \emph{Ergodic Theory of Numbers},
The Carus Mathematical Monographs {\bf 29} (2002).

\bibitem[IK]{[IK]} Iosifescu, M. and Kraaikamp, C. -- \emph{Metrical Theory of
Continued Fractions}. Mathematics and its Applications, 547.
Kluwer Academic Publishers, Dordrecht, 2002.

\bibitem[HK]{[HK]}Hartono, Y., and Kraaikamp, C.  -- \emph{Tong's spectrum
for semi-regular continued fraction expansions}, Far East J.\
Math.\ Sci.\ (FJMS) \textbf{13} (2004), no. 2, 137--165.

\bibitem[HS]{[HS]} Haas, A. and Series, C. -- \emph{Hurwitz constants and
Diophantine approximation on Hecke groups}, J.\ London Math.\
Soc.\ {\bf 34} (1986), 219--234.

\bibitem[JK]{[JK]} Jager, H. and Kraaikamp, C. -- \emph{On the approximation
by continued fractions}, Nederl.\ Akad.\ Wetensch.\ Indag.\ Math.\
\textbf{51} (1989), no. 3, 289--307.

\bibitem[KNS]{[KNS]}Kraaikamp, C.,  Nakada, H. and Schmidt, T.A. -- \emph{On approximation
by Rosen continued fractions}, in preparation (2006).

\bibitem[L]{[L]} Legendre, A. M. -- \emph{Essai sur la th\'{e}orie des nombres}, Paris (1798).

\bibitem[N1]{[N1]} Nakada, H. -- \emph{Continued fractions, geodesic
flows and Ford circles}, Algorithms, Fractals and Dynamics,  (1995), 179--191.

\bibitem[N2]{[N2]} Nakada, H. -- \emph{On the Lenstra constant}, submitted (2006).

\bibitem[R]{[R]} Rosen, D. -- \emph{A class of continued fractions associated with certain
properly discontinuous groups}, Duke Math.\ J.\ {\bf21} (1954),
549--563.

\bibitem[T1]{[T1]} Tong, J.C. -- \emph{Approximation by nearest integer continued fractions},
Math.\ Scand.\ \textbf{71} (1992), no. 2, 161--166.

\bibitem[T2]{[T2]} Tong, J.C. -- \emph{Approximation by nearest integer continued fractions. II},
Math.\ Scand.\ \textbf{74} (1994), no. 1, 17--18.

\end{thebibliography}
\end{document}